\documentclass{article}

\title{On the Lower Bound of Minimizing Polyak-\L ojasiewicz Functions}

\usepackage{bbding}

\author{%
  Pengyun Yue $\qquad$ Cong Fang $\qquad$ Zhouchen Lin \\
  School of Intelligence Science and Technology, Peking University\\
  \texttt{\{yuepy,fangcong,zlin\}@pku.edu.cn} \\
}
\date{}

\usepackage{graphicx}
\usepackage{float}
\usepackage{natbib}

\usepackage{algorithm}
\usepackage{algpseudocode}
\usepackage{amsmath}
\usepackage{amsthm}
\usepackage{changes}
\usepackage{subfig}

\usepackage{amsmath,amsfonts,bm}
\usepackage{xfrac}
\usepackage{algorithm}
\usepackage{algpseudocode}

\newtheorem{theorem}{Theorem}
\newtheorem{lemma}{Lemma}

\newcommand{\R}{\mathbb{R}}
\newcommand{\N}{\mathbb{N}}

\newcommand{\Pdelta}{\mathcal {P}(\Delta, \mu, L)}
\newcommand{\supp}{\mathop{\mathrm{supp}}}
\newcommand{\x}{\mathbf{x}}
\newcommand{\y}{\mathbf{y}}
\newcommand{\z}{\mathbf{z}}

\newcommand{\vT}{v_{T,c}}

\newcommand{\bbT}{\mathbf{b}_{T,c}}
\newcommand{\bT}{b_{T,c}}

\newcommand{\xx}[1]{\mathbf{x}^{(#1)}}

\newcommand{\ee}[1]{\mathbf{e}^{(#1)}}
\newcommand{\EE}[1]{\mathbf{E}_{#1}}

\newcommand{\bb}{\mathbf{b}}

\newcommand{\bI}{\mathbf{I}}
\newcommand{\sA}{\mathsf{A}}

\newcommand{\bA}{\mathbf{A}}

\newcommand{\tg}{\tilde{g}}

\newcommand{\tx}{\tilde{\mathbf x}}

\newcommand{\tI}{\tilde{\mathbf I}}

\newcommand{\Afirst}{\mathcal{A}^{(1)}}
\newcommand{\Als}{\mathcal {A}_{\mathrm{zr}}}

\newcommand{\eps}{\varepsilon}
\newcommand{\argmin}{\mathop{\mathrm{argmin}}}
\newcommand{\argmax}{\mathop{\mathrm{argmax}}}

\newcommand{\<}{\left\langle}
\renewcommand{\>}{\right\rangle}

\newcommand{\modify}[1]{{#1}}

\newcommand{\vy}{v_{y}}
\newcommand{\by}{b_{y}}
\newcommand{\qtt}{q_{T, t}}
\newcommand{\bB}{\mathbf{B}}
\newcommand{\gtt}{g_{T, t}}
\newcommand{\tB}{\tilde{\mathbf{B}}}
\newcommand{\ty}{\tilde{\mathbf{y}}}
\newcommand{\tz}{\tilde{\mathbf{z}}}
\newcommand{\dominate}{\mathsf{Dominate}}
\newcommand{\true}{\mathrm{True}}
\newcommand{\false}{\mathrm{False}}
\newcommand{\Next}{\mathsf{Next}}
\newcommand{\UBindex}{\mathsf{UBIndex}}
\newcommand{\Dist}{\mathsf{Dist}}
\newcommand{\UB}{\mathrm{UB}}

\def\app{1}
\begin{document}

\maketitle

\begin{abstract}
Polyak-\L ojasiewicz (PL) \citep{POLYAK1963864} condition is a weaker condition than the strong convexity but suffices to ensure a global convergence for the Gradient Descent algorithm. In this paper, we study the lower bound of algorithms using first-order oracles to find an approximate optimal solution. We show that any first-order algorithm  requires at least  ${\Omega}\left(\frac{L}{\mu}\log\frac{1}{\eps}\right)$ gradient costs to find an  $\eps$-approximate optimal solution for  a general $L$-smooth function that has an $\mu$-PL constant. This result demonstrates the \textit{optimality} of the Gradient Descent algorithm to minimize smooth PL functions in the sense that there exists a ``hard'' PL function such that no first-order algorithm can be faster than Gradient Descent when ignoring a numerical constant. In contrast, it is well-known that the momentum technique, e.g.  \citet[chap.~2]{nesterov2003introductory}, can provably accelerate Gradient Descent to ${O}\left(\sqrt{\frac{L}{\hat{\mu}}}\log\frac{1}{\eps}\right)$ gradient costs for functions that are $L$-smooth and $\hat{\mu}$-strongly convex. Therefore, our result distinguishes the hardness of minimizing a smooth PL function and a smooth strongly convex function as the complexity of the  former cannot be improved by any polynomial order in general. 

\end{abstract}

\section{Introduction}
We consider the problem
\begin{equation}
    \min_{\x\in \R^d} f(\x),
    \label{equ:problem}
\end{equation}
where the function $f$ is $L$-smooth and satisfies the Polyak-\L ojasiewicz condition. A
function $f$ is said to satisfy the Polyak-\L ojasiewicz condition if \eqref{equ:PLdef} holds for some $\mu>0$:

\begin{equation}
    \|\nabla f(\x)\|^2\ge 2\mu\left(f(\x) - \inf_{\y\in \R^d} f(\y) \right),\quad \forall \x\in \R^d.
    \label{equ:PLdef}
\end{equation}

We  refer to \eqref{equ:PLdef} as the $\mu$-PL condition and simply denote $\inf_{\y\in \R^d} f(\y)$ by $f^*$. The PL condition may be originally introduced by Polyak \modify{\citep{POLYAK1963864} and \L ojasiewicz \citep{lojasiewicz1963topological} independently.}
The PL condition is strictly weaker than strong convexity as one can show that any $\hat{\mu}$-strongly convex function \modify{which by definition satisfies}:
$$ f(\x) \geq f(\y) + \<\nabla f(\y), \x - \y \> + \frac{\hat{\mu}}{2} \left\|\x- \y \right\|^2 $$
is also  $\hat{\mu}$-PL by minimizing both sides with respect to $\x$ \citep{karimi2016linear}. However, the PL condition does not even imply convexity.  From a geometric view, the PL condition suggests that the sum of the squares of the gradient dominates the optimal function value gap, which implies that any local stationary point is a global minimizer. Because it is relatively easy to obtain an approximate local stationary point by first-order algorithms, the PL condition serves as an ideal and weaker alternative to strong convexity.

In machine learning, the PL condition has received wide attention recently. Lots of models are found to satisfy this condition under different regimes. Examples include, but are not limited to, matrix decomposition and linear neural networks under a  specific initialization \citep{hardt2016identity, li2018algorithmic}, nonlinear neural networks in the so-called neural tangent kernel regime \citep{liu2022loss}, reinforcement learning with linear quadratic regulator \citep{fazel2018global}. Compared with strong convexity, the PL condition is much easier to hold since the reference point in the latter only is a minimum point such that $\x^*=\argmin_{\y} f(\y)$, instead of any $\y$ in the domain. 

Turning to the theoretic side, it is known \citep{karimi2016linear} that the standard Gradient Descent algorithm admits a linear converge to minimize a $L$-smooth and $\mu$-PL function. To be specific, in order to find an $\eps$-approximate optimal solution $\hat{\x}$ such that $f(\hat{\x}) - f^* \le \eps$, Gradient Decent needs  $O(\frac{L}{\mu}\log\frac{1}{\eps})$ gradient computations.  However, it is still not clear whether there exist algorithms that can achieve a provably faster convergence rate. In the optimization community, it is perhaps well-known that the momentum technique, e.g. \citet[chap.~2]{nesterov2003introductory}, can provably accelerate Gradient Descent from $O(\frac{L}{\hat{\mu}}\log \frac{1}{\eps})$ to $O\left(\sqrt{\frac{L}{\hat{\mu}}}\log \frac{1}{\eps}\right)$ for  functions that are $L$-smooth and $\hat{\mu}$-strongly convex. Even though some works \citep{j2016proximal, lei2017non} have considered accelerations under different settings, probably faster convergence of first-order algorithms for PL functions is still not obtained up to now.

In this paper,  we study the first-order complexities to minimize a generic smooth PL function and ask the question:

\textit{``Is the Gradient Decent algorithm (nearly) optimal  or can we design a much faster algorithm?''}

We answer the question in the language of min-max lower bound complexity for minimizing the $L$-smooth and $\mu$-PL function class. We  analyze the worst complexity of minimizing any function that belongs to the  class using first-order algorithms.  Excitingly,   we construct  a hard instance function showing that any first-order algorithm  requires at least  ${\Omega}\left(\frac{L}{\mu}\log\frac{1}{\eps}\right)$ gradient costs to find an $\eps$-approximate optimal solution.  This answers the aforementioned question in an explicit way:  the Gradient Descent algorithm is already \textit{optimal} in the sense that no first-order algorithm can achieve a provably faster convergence rate in general ignoring a numerical constant.  For the first time, we  distinguish the hardness of minimizing a  PL function and a strongly convex function in terms of first-order complexities, as the momentum technique for smooth and strongly convex functions provably accelerates Gradient Descent by a certain polynomial order.

It is worth mentioning that 
the optimization problem under our consideration is high-dimensional and the goal is to obtain the complexity bounds that do not have an explicit dependency on the dimension. 

Our technique to establish the lower bound follows from the previous lower bounds in convex \citep{nesterov2003introductory} and non-convex optimization \citep{carmon2021lower}. The main idea is to construct a so-called ``zero-chain'' function ensuring that any first-order algorithm per-iteratively can only solve one coordinate of the optimization variable.  Then for  a  ``zero-chain'' function that has a sufficiently high dimension,  some number of entries will never reach their optimal values after the execution of any first-order algorithm in certain iterations.  To obtain the desired ${\Omega}\left(\frac{L}{\mu}\log\frac{1}{\eps}\right)$ lower bound, we propose a   ``zero-chain'' function similar to \citet{carmon2020lower},  which is composed of  the worst convex function designed by \citet{nesterov2003introductory} and a  separable function in the form as $\sum_{i=1}^T v_{\y_i}(\x_i)$ to destroy the convexity. Different from their separable function, the one that we introduce has a large Lipshictz constant. This property helps us to estimate the PL constant in a convenient way. This new idea gives new insights into the constructions and analyses of instance functions, which might be potentially generalized to establish the lower bounds for other non-convex problems. 

\subsection*{Notation}
We use bold letters, such as $\x$, to denote vectors in the Euclidean space $\R^d$, and bold capital letters, such as $\bA$, to denote matrices. $\bI_d$ denotes the identity matrix of size $d\times d$.  We omit the subscript and simply denote $\bI$ as the identity matrix when the dimension is clear from context. For $\x\in \R^d$, we use $\x_i$ to denote its $i$th coordinate. We use $\supp(\x)$ to denote the subscripts of non-zero entries of $\x$, i.e. $\supp(\x) = \{i:\x_i\ne 0\}$. We use $\mathop{\mathrm{span}}\left\{\xx 1,\cdots,\xx n \right\}$ to denote the linear subspace spanned by $\xx 1,\cdots,\xx n$, i.e. $\left\{\y:\y = \sum_{i=1}^n a_i\xx i, a_i\in \R\right\}$. We call a function $f$ $L$-smooth if $\nabla f$ is $L$-Lipschitz continuous, i.e. $\|\nabla f(\x)-\nabla f(\y)\|\le L\|\x-\y\|$. We denote $f^* = \inf_{\x} f(\x)$.
We let $\x^*$ be any minimizer of $f$, i.e.,  $\x^* = \argmin f$.  We always assume the existence of $\x^*$. We say that $\x$ is an $\eps$-approximate optimal point of $f$ when $f(\x)-f^*\le \eps$. 

\section{Related Work}

\textbf{Lower Bounds} There has been a line of research concerning the lower bounds of algorithms on certain function classes. To the best of our knowledge, \citep{nemirovskij1983problem} defines the oracle model to measure the complexity of algorithms, and most existing research on lower bounds follow this formulation of complexity. For convex functions and first-order oracles, the lower bound is studied in \citet{nesterov2003introductory}, where well-known optimal lower bound $\Omega(\eps^{-\frac12})$ and $\Omega(\kappa\log\frac{1}{\eps})$ are obtained. For convex functions and $n$th-order oracles, lower bounds $\Omega\left(\eps^{-\frac{2}{3n+1}}\right)$ have been proposed in \citet{arjevani2019oracle}. When the function is non-convex, it is generally NP-hard to find its global minima, or to test whether a point is a local minimum or a saddle point \citep{murty1985some}. Instead of finding $\eps$-approximate optimal points, an alternative measure is finding $\eps$-stationary points where $\|\nabla f(\x)\|\le\eps$. Sometimes, additional constraints on the Hessian matrices of second-order stationary points are needed. Results of this kind include \citet{carmon2020lower,carmon2021lower,fang2018spider,zhou2019lower,arjevani2019lower,arjevani2020second}. Though a PL function may be non-convex, it is tractable to find an  $\eps$-approximate optimal point, as local minima of a PL function must be global minima. In this paper, we give the lower complexity bound for finding $\eps$-approximate optimal points.

\noindent\textbf{PL Condition} The PL condition was introduced by Polyak \citep{POLYAK1963864} and \L ojasiewicz \citep{lojasiewicz1963topological} independently.  Besides the PL condition, there are other relaxations of the strong convexity, including error bounds \citep{luo1993error}, essential strong convexity \citep{liu2014asynchronous}, weak strong convexity \citep{necoara2019linear}, restricted secant inequality \citep{zhang2013gradient}, and quadratic growth \citep{anitescu2000degenerate}. \citet{karimi2016linear} discussed the relationships between these conditions. All these relaxations implies the PL condition except for the quadratic growth, which implies that the PL condition is quite general. \citet{danilova2020non} studied the convergence rate of Heavy-ball method on PL functions. \citet{wang2022provable} proved an accelerated convergence rate for Heavy-ball algorithm when the non-convexity is ``averaged-out''. There are many other papers that study designing practical algorithms to optimize a PL objective function under different scenarios, for example, \citet{bassily2018exponential,nouiehed2019solving,hardt2016identity,fazel2018global,j2016proximal,lei2017non}.

\section{Preliminaries}
\subsection{Upper bound on PL functions}

Although the PL condition is a weaker condition than strong convexity, it guarantees linear convergence for Gradient Descent. The result can be found in \citet{POLYAK1963864} and \citet{karimi2016linear}. We present it here for completeness.

\newcounter{tapp1}
\setcounter{tapp1}{\value{theorem}}
\begin{theorem}
    If $f$ is $L$-smooth and satisfies $\mu$-PL condition, then the Gradient Descent algorithm with a constant step-size $\frac1L$:
    \begin{equation}
        \xx{k+1} = \xx{k} - \frac1L \nabla f(\xx{k}),
        \label{equ:GD}
    \end{equation}
    has a linear convergence rate. We have:
    \begin{equation}
        f(\xx{k})-f^* \le \left(1-\frac\mu L \right)^k(f(\xx{0}) - f^*).
    \end{equation}
    \label{thm:linear}
\end{theorem}


Theorem \ref{thm:linear} shows that the Gradient Descent algorithm finds the $\eps$-approximate optimal point of $f$ in $O\left(\frac{L}{\mu}\log \frac 1\eps \right)$ gradient computations. This gives an upper complexity bound for first-order algorithms. However, it remains open to us whether there are faster algorithms for smooth PL functions. We will establish a lower complexity bound on first-order algorithms, which nearly matches the upper bound.

\subsection{Definitions of algorithm classes and function classes}
An algorithm is a mapping from real-valued functions to sequences. For algorithm $\sA$ and $f:\R^d\to \R$, we define $\sA[f]=\{\xx{i}\}_{i\in \N}$ to be the sequence of algorithm $\mathsf A$ acting on $f$, where $\xx{i}\in \R^d$. 

Note here, the algorithm under our consideration works on function defined on any Euclidean space. 
We call it the dimension-free property of the algorithm.

The definition of algorithms abstracts away from the the optimization process of a function. We consider algorithms which only make use of the first-order information of the iteration sequence. We call them first-order algorithms. If an algorithm is a first-order algorithm, then 
\begin{equation}
    \xx{i} = \mathsf A^{(i)} \left(\xx{0},\nabla f(\xx{0}),\cdots,\xx{i-1}, \nabla f(\xx{i-1}) \right),
    \label{equ:first-order}
\end{equation}
where $\mathsf A^{(i)}$ is a function depending on $\mathsf A$. Perhaps the simplest example of first-order algorithms is Gradient Descent.

We are interested in finding an $\eps$-approximate point of a function $f$. Given a function $f:\R^d\to \R$ and an algorithm $\mathsf A$, the complexity of $\mathsf A$ on $f$ is the number of queries to the first-order oracle needed to find an $\eps$-approximate point. We denote $T_\eps(\mathsf A, f)$ to be the gradient complexity of $\sA$ on $f$, then 
\begin{equation}
    T_\eps(\mathsf A, f) = \min_t\left\{t:f(\xx{t})-f^*\le \eps \right\}.
\end{equation}

In practice, we do not have the full information of the function $f$. We only know that $f$ is in a particular function class $\mathcal F$, such as $L$-smooth functions. Given an algorithm $\mathsf A$. We denote $\mathsf T_\eps(\mathsf A,\mathcal F)$ to be the complexity of $\mathsf A$ on $\mathcal F$, and define $\mathsf T_\eps(\mathsf A,\mathcal F)$ as follows:
\begin{equation}
    \mathsf T_\eps(\mathsf A,\mathcal F) = \sup_{f\in \mathcal F} T_\eps(\mathsf A,f).
\end{equation}
Thus, $\mathsf T_\eps(\mathsf A, \mathcal F)$ is the worst-case complexity of functions $f\in \mathcal F$.

For searching an $\eps$-approximate optimal point of a function in $\mathcal F$, we need to find an algorithm which have a low complexity on $\mathcal F$.  Denote an algorithm class by $\mathcal A$. The lower bound of an algorithm class on $\mathcal F$ describes the efficiency of algorithm class $\mathcal A$ on function class $\mathcal F$, which is defined to be
\begin{equation}
    \mathcal T_\eps(\mathcal A,\mathcal F) = \inf_{\mathsf A\in \mathcal A} \mathsf T_\eps(\mathsf A, \mathcal F) = \inf_{\mathsf A\in\mathcal A}\sup_{f\in \mathcal F} T_\eps(\mathsf A,f).
\end{equation}

\subsection{Zero-respecting Algorithm}
Among all the algorithms, a special algorithm class is called zero-respecting algorithms. If $\mathsf A$ is a zero-respecting algorithm and $\sA[f] = \left\{\xx t\right\}_{t\in \N}$, then the following condition holds for all $f:\R^d\to\R$:
\begin{equation}
    \supp\{\xx n-\xx 0\} \in \bigcup_{i=1}^{n-1} \supp \{\nabla f(\xx i)\}.
\end{equation}
Note that if $\xx n-\xx 0$ lies in the linear subspace spanned by $\nabla f(\xx 0),\cdots,\nabla f(\xx{n-1})$, then $\mathsf A$ is a zero-respecting algotithm. We denote the collection of first-order zero-respecting algorithms with $\xx 0 = \mathbf 0$ by $\Als$. It is shown by \citet{nemirovskij1983problem} that a lower complexity bound on first-order zero-respecting 
algorithms are also a lower complexity bound on all the first-order algorithm when the function class satisfies the orthogonal invariance property.

\subsection{Zero-chain}
A zero-chain $f$ is a function that safisfies the following condition:
\begin{equation}
    \supp (\x)\subseteq \{1,2,\cdots,k\}\implies \supp (\nabla f(\x))\subseteq \{1,2,\cdots, k+1\},\quad \forall \x.
\end{equation}
In other words, the support of $\nabla f(\x)$ lies in a restricted linear subspace depending on the support of $\x$. 

The ``worst function in the (convex) world'' in \citet{nesterov2003introductory} defined as 
\begin{equation}
    f_d(\x) = \frac12 (\x_1-1)^2 + \sum_{i=1}^{d-1} (\x_{i+1}-\x_i)^2 
    \label{equ:Nesdifficult}
\end{equation}
is a zero-chain, because if $\x_i = 0$ for $i>n$, then $\left(\nabla f_d(\x)\right)_{i+1} = 0$ for $i>n$. A zero-chain is difficult to optimize for zero-respecting algorithms, because zero-respecting algorithms only discover one coordinate by one gradient computation. 

\section{Main results}
\label{sec:main_result}
According to Theorem \ref{thm:linear}, we already have an upper complexity bound $O\left(\frac L\mu \log\frac 1\eps\right)$ by applying Gradient Descent to all the PL functions. In this section, we establish the lower complexity bound of first-order algorithms on PL functions. 
Let $\Pdelta$ be the collection of all $L$-smooth and $\mu$-PL functions $f$ with $f(\xx 0) - f^* \le \Delta$. We establish a lower bound of $\mathcal T_\eps\left(\Als, \Pdelta \right) $ by constructing a function which is hard to optimize for zero-respecting algorithms, and extend the result to first-order algorithms. We present a hard instance 
 that 
 can achieve the desired $\Omega\left(\kappa\log\frac{1}{\eps}\right)$ lower bound below. 



We first introduce several components of the hard instance. For the non-convex part, we define 
\begin{equation}
    \vy(x) = \begin{cases}
        \frac12 x^2,&x\le \frac{31}{32}y,\\
        \frac12 x^2 - 16\left(x-\frac{31}{32}y \right)^2, &\frac{31}{32}y< x\le y,\\
        \frac12 x^2 - \frac{y^2}{32} + 16\left(x-\frac{33}{32}y\right)^2, &y<x\le \frac{33}{32}y,\\
        \frac12 x^2-\frac{y^2}{32}, &x>\frac{33}{32}y,
    \end{cases}
    \label{equ:vydef}
\end{equation}
where $y>0$ is a constant. By the definition of $\vy$, we have
\begin{equation}
    \vy'(x) = \begin{cases}
        x, &x\le \frac{31}{32}y,\\
        x - 32\left(x-\frac{31}{32}y \right), &\frac{31}{32}y< x\le y,\\
        x + 32\left(x-\frac{33}{32}y\right), &y<x\le \frac{33}{32}y,\\
        x, &x>\frac{33}{32}y.
    \end{cases}
\end{equation}
Define
\begin{equation}
    \by(x) = \begin{cases}
        y-32|x-y|, &\frac{31}{32}y\le x\le \frac{33}{32}y,\\
        0, &\text{otherwise}.
        \end{cases}
        \label{equ:bydef}
\end{equation}
Then we have $\vy'(x) = x-\by(x)$. 

For the convex part, we define $\qtt(\x)$ as follows (for the convenience of notation, we define $\x_0 = 0$):
\begin{equation}
    \qtt(\x) = \frac12\sum_{i=0}^{t-1} \left[\left(\frac{7}{8}\x_{iT}-\x_{iT+1}\right)^2 +  \sum_{j=1}^{T-1} \left(\x_{iT+j+1}-\x_{iT+j}\right)^2\right],
\end{equation}
where $\x\in \R^{Tt}$. $\qtt$ is a quadratic function of $\x$, thus can be written as 
\begin{equation}
    \qtt(\x) = \frac12 \x^T\bB \x,
\end{equation}
where $\bB$ is a positive semi-definite symmetric matrix. $\bB$ satisfies $0\preceq \bB\preceq 4\bI$, because the sum of absolute value of non-zero entries of each row of $\bB$ is smaller or equal to $4$.

The quadratic part $q$ is very similar to ``the worst function in the (convex) world'' in \citet{nesterov2003introductory}, and the definition of $\vy$ is inspired by the hard instance in \citet{carmon2021lower}. Our hard instance differs from previous ones mainly in the large Lipschitz constant of its gradient. We note that the controlled degree of nonsmoothness is crucial for our estimate of PL constant.  

Let $\y\in \R^{Tt}$ be a vector satisfying $\y_{qT+b} = {\left(\frac{7}{8}\right)}^{q}$, where $q\in \N$, $b\in \{1, 2, \cdots, T\}$. We define the hard instance $\gtt:\R^{Tt}\to \R$ as follows:
\begin{equation}
    \gtt(\x) = \qtt(\x) + \sum_{i=1}^{Tt} v_{\y_i}(\x_i).
    \label{equ:gttdef}
\end{equation}

Now we list some properties of $\gtt$ in Lemma \ref{lem:gttprop}, which we prove in Appendix \ref{subsec:logepsproof}.
\newcounter{lem4}
\setcounter{lem4}{\value{theorem}}
\begin{lemma}
    $\gtt$ satisfies the following.
    \begin{enumerate}
        \item $\gtt(\y-\x)$ is a zero-chain.
        \label{gprop0:zc}
        \item $\x^* = \mathbf 0$, $\gtt^* = 0$, $\gtt(\x)\le \frac12\x^T(\bB+\bI)\x$.
        \label{gprop1:min}
        \item $\gtt$ is $37$-smooth.
        \label{gprop2:smooth}
        \item $\gtt$ satisfies the $\frac{1}{C_3T}$-PL condition, where $C_3$ is a universal constant.
        \label{gprop3:PL}
    \end{enumerate}
    \label{lem:gttprop}
\end{lemma}

Define $\tg$ to be the following function, which is hard for first-order algorithms:
\begin{equation}
    \tg(\x) = \frac{L T^{-1} D^2}{37}\gtt\left(\y - T^{1/2}D^{-1}\x\right),
    \label{equ:tgref}
\end{equation}
where $D=c\|\x^{(0)}-\x^*\|$, and $c$ is a constant. In smooth optimization, $D$ is often treated as a constant.

In Lemma \ref{lem:zero-chain-log} below, we show that $\tg$ is hard for first-order zero-respecting algorithms:
\newcounter{lem5}
\setcounter{lem5}{\value{theorem}}
\begin{lemma}
    Assume that $\eps < 0.01$ and let $t = 2\left\lfloor\log_{\frac{8}{7}} \frac{3}{2\eps}\right\rfloor$. A first-order zero-respecting algorithm with $\xx 0 = \mathbf 0$ needs at least $\frac12Tt$ gradient computations to find a point $\x$ satisfying $\tg(\x) - \tg^*\le \eps(\tg (\xx 0) - \tg^*)$.
    \label{lem:zero-chain-log}
\end{lemma}
\begin{proof}
    By induction, we have $\supp(\xx k)\subseteq \{1,\cdots,k\}$. By the definition of $\tg$ and $\vy$, we have
    \begin{equation}
        \begin{aligned}
        \gtt(\mathbf 0) &= \frac12 + \sum_{i=0}^{t-1}Tv_{\left(\frac{7}{8}\right)^i}\left(\left(\frac{7}{8}\right)^i\right)\\
        &= \frac12 + T\sum_{i=0}^{t-1} \frac{31}{64}\left(\frac{7}{8}\right)^{2i}\\
        &= \frac12 + T\cdot\frac{31}{15}\cdot\left(1-\left(\frac{7}{8}\right)^{2t}\right)\\
        &\le 3T.\\
        \end{aligned}
    \end{equation}
    For $k\le \frac12Tt$,
    \begin{equation}
        \begin{aligned}
            \gtt(\xx k)&\ge \sum_{i=\frac{t}{2}}^t Tv_{\left(\frac{7}{8}\right)^i}\left(\left(\frac{7}{8}\right)^i\right)\\
            &=T\sum_{i=\frac{t}{2}}^t \frac{31}{64}\left(\frac{7}{8}\right)^{2i}\\
            &= T\cdot \left(\frac{7}{8}\right)^{\frac{t}{2}}\cdot \frac{31}{15}\left(1-\left(\frac{7}{8}\right)^{t}\right)\\
            &\ge 2T\cdot \frac{3\eps}{2}\\
            &= 3T\eps.
        \end{aligned}
    \end{equation}
    Therefore, for $k\le \frac{1}{2}Tt$, $\tg(\xx k) - \tg^*\ge \eps(\tg (\xx 0) - \tg^*)$.
\end{proof}

With Lemma \ref{lem:gttprop} and \ref{lem:zero-chain-log}, we obtain a lower bound for zero-respeting algorithms:
\newcounter{thm4}
\setcounter{thm4}{\value{theorem}}
\begin{theorem}
    Given $L\ge\mu>0$. When $\kappa = \frac{L}{\mu}> C_4$ where $C_4$ is a universal constant, there exists $T$ and $t$ such that $\tg$ is $L$-smooth and $\mu$-PL. Moreover, any first-order zero-respecting algorithm with $\xx{0} = \mathbf 0$ needs at least $\Omega\left(\kappa\log\frac{1}{\eps}\right)$ gradient computations to find a point $\x$ satisfying $\tg(\x) - \tg^*\le \eps(\tg (\xx 0) - \tg^*)$. 
    \label{thm:mainlog}
\end{theorem}

\begin{proof}
    We let $C_4=370C_3$. Given $\kappa>C_4$, we let 
    \begin{equation}
        T = \left\lfloor\frac{\kappa}{37C_3} \right\rfloor,
        \label{equ:gTdef}
    \end{equation}
    and 
    \begin{equation}
        t = 2\left\lfloor\log_{\frac{8}{7}} \frac{3}{2\eps}\right\rfloor.
        \label{equ:gtdef}
    \end{equation}
    
    We use Lemma \ref{lem:gttprop} to calculate the smoothness constant and PL constant of $\tg$. The smoothness constant of $\tg$ is:
    \begin{equation}
        \begin{aligned}
            \frac{L T^{-1} D^2}{37} \cdot T D^{-2} \cdot 37 = L,
        \end{aligned}
    \end{equation}
    and the PL constant of $\tg$ is:
    \begin{equation}
        \begin{aligned}
            \frac{L T^{-1} D^2}{37} \cdot T D^{-2} \cdot \frac{1}{C_3T}&=\frac{L}{37C_3T}\\
            &\stackrel{a}{=} \frac{L}{37C_3} \cdot\frac{1}{\left\lfloor\frac{\kappa}{37C_3} \right\rfloor}\\
            &\ge \frac{L}{\kappa}\\
            &=\mu,
        \end{aligned}
    \end{equation}
    where $\stackrel{a}{=}$ uses \eqref{equ:gTdef}.
    
    Finally, by Lemma \ref{lem:zero-chain-log}, any first-order zero-respecting algorithm needs at least $\frac {Tt}2$ accesses to gradient to find $\x$ such that $\tg(x)-\tg^*\le \eps\left(\tg(\xx 0) - \tg^* \right)$. By \eqref{equ:gTdef} and \eqref{equ:gtdef},
    \begin{equation}
        \frac {Tt}{2}\ge \left(\frac{\kappa}{37C_1}-1\right)\cdot\left(\log_{\frac{8}{7}}\frac{3}{2\eps}\right) = \Omega\left(\kappa\log\frac{1}{\eps}\right),
    \end{equation}
    which completes the proof.
\end{proof}

Using the technique of \citet{nemirovskij1983problem}, for specific function classes such as PL functions, a lower complexity bound on first-order zero-respecting algorithms is also a lower complexity bound on all the first-order algorithms. Denoting the set of all first-order algorithms by $\Afirst$, we have the following lemma:

\newcounter{app3}
\setcounter{app3}{\value{theorem}}
\begin{lemma}
    \begin{equation}
        \mathcal T_\eps\left(\Afirst, \Pdelta\right)\ge \mathcal T_\eps\left(\Als, \Pdelta\right).
    \end{equation}
    \label{lem:ls}
\end{lemma}
\begin{proof}
    The set $\Pdelta$ satisfies orthogonal invariance property \citep{bassily2018exponential}. Therefore, the results follows from Proposition 1 of \citet{carmon2020lower}.
\end{proof}

Finally, we arrive at a lower bound for first-order algorithms:

\newcounter{thm5}
\setcounter{thm5}{\value{theorem}}
\begin{theorem}
    For any $0<a<1$, when $\eps \le \frac{1}{16}\Delta$, 
    \begin{equation}
        \mathcal T_\eps\left(\Afirst, \Pdelta \right)\ge \Omega\left(\kappa\log\frac{1}{\eps}\right).
        \label{equ:Afirstbound-tight}
    \end{equation}
\end{theorem}
\begin{proof}
    The result is a direct corollary of Theorem \ref{thm:mainlog} and Lemma \ref{lem:ls}.
\end{proof}
This bound matches the convergence rate of Gradient Descent up to a constant.

\section{Numerical experiments}
\label{sec:numerical_experiments}
\begin{figure}[t]
    \centering
    \subfloat[$\kappa = 1.9709\times 10^6, \eps=10^{-10}$]{
      \label{sfig:hard}
      \includegraphics[height=0.2\textheight]{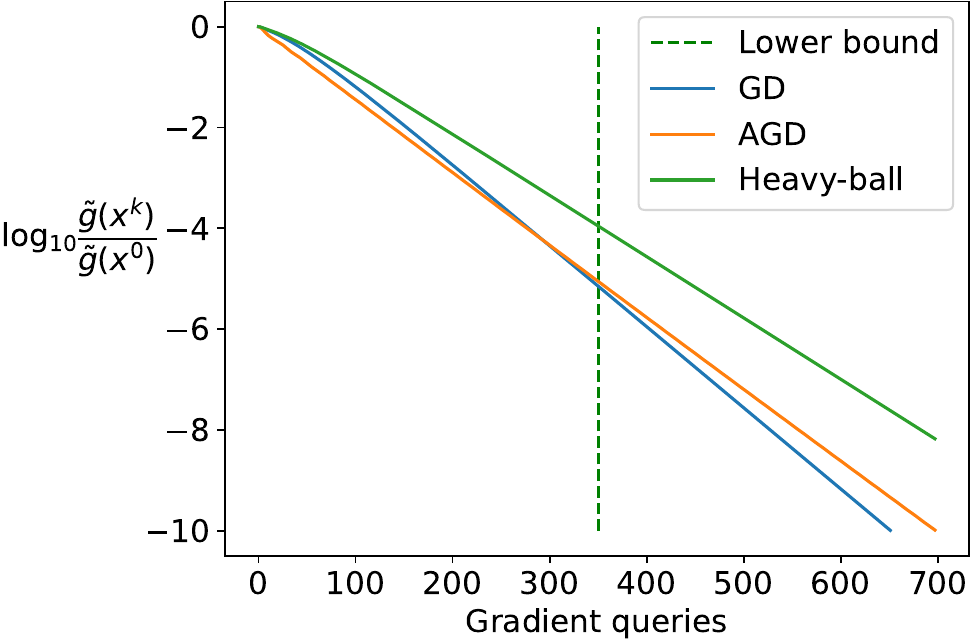}}
    \subfloat[$\kappa = 7.3119\times 10^{6}, \eps=10^{-20}$]{
      \label{sfig:hard2}
      \includegraphics[height=0.2\textheight]{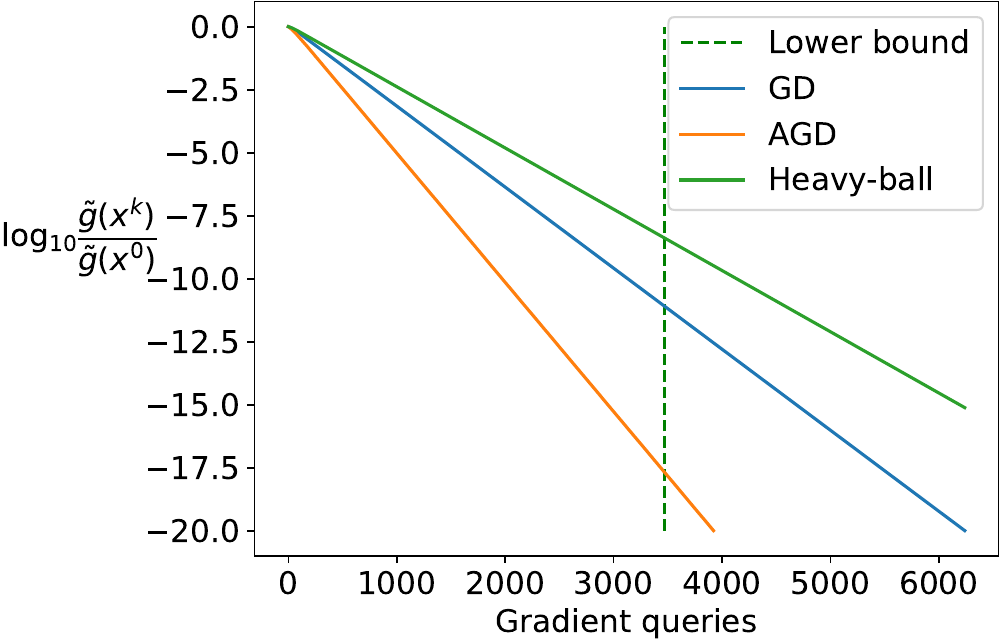}}
    \caption{Convergence rate under Gradient Descent, Nesterov's Accelerated Gradient Descent and Polyak's Heavy-ball Method}
    \label{fig:numerical}
\end{figure}
 We conduct numerical experiments  on our hard instance. We consider the $\kappa$ relatively large, which can reduce the factors from the numerical constants. We first choose $\kappa$ and $\eps$, and then decide $T$ and $t$ using \eqref{equ:gTdef} and \eqref{equ:gtdef}. We use Gradient Descent, Nesterov's Accelerated Gradient Descent (AGD) and Polyak's Heavy-ball Method to optimize the hard instance. As AGD and the Heavy-ball Method are designed for  convex functions, we need to choose appropriate parameter $\hat\mu$ in both algorithms, because our hard instance is non-convex. For AGD, We let $\hat\mu = \mu$, the PL constant of our hard instance. For Heavy-ball method, we adopt the parameter setting in \citep{danilova2020non}.

GD, AGD and Heavy-ball Method are all zero-respecting algorithms, so Lemma \ref{lem:zero-chain-log} and Theorem \ref{thm:mainlog} applies to their convergence rates. From Figure \ref{fig:numerical}, we observe that all three algorithms converge almost linearly, but the number of  greadient queries is more than the complexity lower bound. The result is consistent with Lemma \ref{lem:zero-chain-log} and Theorem \ref{thm:mainlog}.

\section{Conclusion}
We construct a lower complexity bound on optimizing smooth PL functions with first-order methods. A first-order algorithm needs at least $\Omega\left(\frac L\mu\log\frac{1}{\eps}\right)$ gradient access to find an $\eps$-approximate optimal point of an $L$-smooth $\mu$-PL function. Our lower bound matches the convergence rate of Gradient Descent up to constants. 

We only focus on deterministic algorithms in this paper.  We conjecture  that our results can be extended to randomized algorithms, using the same technique in \citet{nemirovskij1983problem} and explicit construction in \citet{woodworth2016tight} and \citet{woodworth2017lower}. We leave its formal derivation to the future work.

\subsubsection*{Acknowledgments}
C. Fang and Z. Lin were supported by National Key R\&D Program of China (2022ZD0160302). Z. Lin was also supported by the NSF China (No. 62276004), the major key project of PCL, China (No. PCL2021A12) and Qualcomm. C. Fang was also supported by Wudao Foundation. Thanks for  helpful discussions with Tong Zhang and Weijie Su.

\bibliographystyle{abbrvnat} 

\bibliography{ref}

\if\app1

\appendix

\label{appendix}

\section{Proof of Theorem \ref{thm:linear}}

\newcounter{tt1}
\setcounter{tt1}{\value{theorem}}
\setcounter{theorem}{\value{tapp1}}
\begin{theorem}
    If $f$ is $L$-smooth and satisfies $\mu$-PL condition, then the Gradient Descent algorithm with a constant step-size $\frac1L$:
    \begin{equation}
        \xx{k+1} = \xx{k} - \frac1L \nabla f(\xx{k}),
    \end{equation}
    has a linear convergence rate. We have:
    \begin{equation}
        f(\xx{k})-f^* \le \left(1-\frac\mu L \right)^k(f(\xx{0}) - f^*).
    \end{equation}
\end{theorem}
\setcounter{theorem}{\value{tt1}}

\begin{proof}
    The proof is taken from \citet{karimi2016linear}.
    By the $L$-smoothness of $f$,
    \begin{equation}
        f(\xx{k+1})\le f(\xx{k})+\left\langle\nabla f(\xx k),\xx{k+1}-\xx{k} \right\rangle+\frac L2 \left\|\xx{k+1}-\xx{k} \right\|^2.
        \label{equ:thm1pf1}
    \end{equation}
    Applying \eqref{equ:GD} and \eqref{equ:PLdef} in \eqref{equ:thm1pf1}, we have
    \begin{equation}
        f(\xx{k+1})-f(\xx{k})\le \frac{1}{2L} -\left\| \nabla f(\xx{k})\right\|^2\le -\frac{\mu}{L}\left(f(\xx{k})-f^*\right).
        \label{equ:thm1pf2}
    \end{equation}
    Rearranging terms in \eqref{equ:thm1pf2}, we have
    \begin{equation}
        f(\xx{k+1})-f^*\le \left(1-\frac{\mu}{L}\right)\left(f(\xx{k})-f^* \right).
        \label{equ:thm1pf3}
    \end{equation}
    Applying \eqref{equ:thm1pf3} recursively gives the result.
\end{proof}

\section{Omitted proof in Section \ref{sec:main_result}}
\label{subsec:logepsproof}


\subsection{Proof of Lemma \ref{lem:gttprop}}

\label{subsec:gttproof}
\newcounter{lem4temp}
\setcounter{lem4temp}{\value{theorem}}
\setcounter{theorem}{\value{lem4}}
\begin{lemma}
    $\gtt$ satisfies the following.
    \begin{enumerate}
        \item $\gtt(\y-\x)$ is a zero-chain.
        \item $\x^* = \mathbf 0$, $\gtt^* = 0$, $\gtt(\x)\le \frac12\x^T(\bB+\bI)\x$.
        \item $\gtt$ is $37$-smooth.
        \item $\gtt$ satisfies the $\frac{1}{C_3T}$-PL condition, where $C_3$ is a universal constant.
    \end{enumerate}
\end{lemma}
\setcounter{theorem}{\value{lem4temp}}

\begin{proof}
    In the proof of Lemma \ref{lem:gttprop}, we define
    \begin{equation}
        \bb(\x) = \begin{bmatrix}b_{\y_1}(\x_1)\\\vdots\\b_{\y_{Tt}}(\x_{Tt})\end{bmatrix}.
    \end{equation}
    \begin{enumerate}
        \item We have
        \begin{equation}
            \nabla \gtt(\y-\x) = -(\bB(\y-\x)+(\y-\x)-\bb(\y-\x)).
        \end{equation} 
        When $\x_i = \cdots = \x_{Tt} = 0$, $b_{\y_j}(\y_j - \x_j) = \y_j$ for $j\ge i$. When $k\ge i+1$, $(\nabla \gtt(\x))_k = (\bB(\y-\x))_k = 0$. Therefore, $\supp \left(\nabla \gtt(\y-\x)\right)\in \{1,2,\cdots,i+1\}$, which implies that $\gtt(\y-\x)$ is a zero-chain with respect to $\x$.
        \item $\qtt(\x)$ attains its minimum at $\x = \mathbf 0$, and $\vy(x)$ attains its minimum at $x=0$. Therefore, $\gtt(\x) = \qtt(\x) +\sum_{t=1}^{Tt} v_{\y_i}(\x_i)$ attains its minimum at $\x = \mathbf 0$, and $\gtt^* = 0$. From the definition of $\vT$ in \eqref{equ:vydef}, we have $\vy(x)\le \frac12 x^2$, which implies $\gtt(\x)\le \frac12\x^T(\bB+\bI)\x$.
        \item Let 
        \begin{equation}
            v(\x) = \sum_{i=1}^{Tt} v_{\y_i}(\x_i).
        \end{equation}
        For $\x_1,\x_2\in \R^{Tt}$, 
        \begin{equation}
            \begin{aligned}
                \|\nabla  v(\x_1)-\nabla v(\x_2)\| &= \|\x_1-\bbT(\x_1)-\x_2+\bbT(\x_2)\|\\
                &\le \|\x_1-\x_2\|+\|\bbT(\x_1)-\bbT(\x_2)\|\\
                &\le \modify{33}\|\x_1-\x_2\|.
            \end{aligned}
            \label{equ:glipschitz}
        \end{equation}
        The last inequality of \eqref{equ:glipschitz} is due to the definition of $\bT$, which implies that $\bT$ is \modify{$32 $}-Lipschitz. Consequently, $\gtt$ is \modify{$37$-smooth} because $\bB\preceq\modify{4}\bI$.
        \item The  PL constant of $\gtt$ can be written as
        $$
            \mu = \inf_{\x\in \R^{Tt}} \frac{\|\nabla \gtt(\x)\|^2}{2(\gtt(x)-\gtt^*)} \ge \inf_{\x\in \R^{Tt}} \frac{\|\nabla \gtt(\x)\|^2}{5\|\x\|^2}.
        $$
        For any $\x$, we estimate $\mu$ by dividing the indices of $\x$ into three sets $A,B,C$ using Lemmas \ref{lem:dom} and \ref{lem:into1}: $A$ contains $k$ such that $|(\nabla \gtt(\x))_k|\ge 0.19|\x_k|$; $B$ contains ``exponential growing chains''; $C$ contains ``flat areas in $\left[\frac{31}{32},\frac{33}{32}\right]$''. Intuitively, if $k\in A$, then $|(\nabla \gtt(\x))_k|^2$ is  large enough, and it can be used to upper bound $|\x_k|^2$ and the norm of ``exponential growing chains'' and ``flat areas in $\left[\frac{31}{32},\frac{33}{32}\right]$'' next to $k$ (with Lemma 10). In the set $B$, $\Next$ defines the increasing direction of ``exponential growing chains''. We use $|(\nabla \gtt(\x))_{\UBindex(k)}|^2$ to upper bound $|\x_k|^2$. The intuitions are shown in Figure \ref{fig:illustration}.

        \begin{figure}[t]
        \centering
        \includegraphics[width=0.95\textwidth]{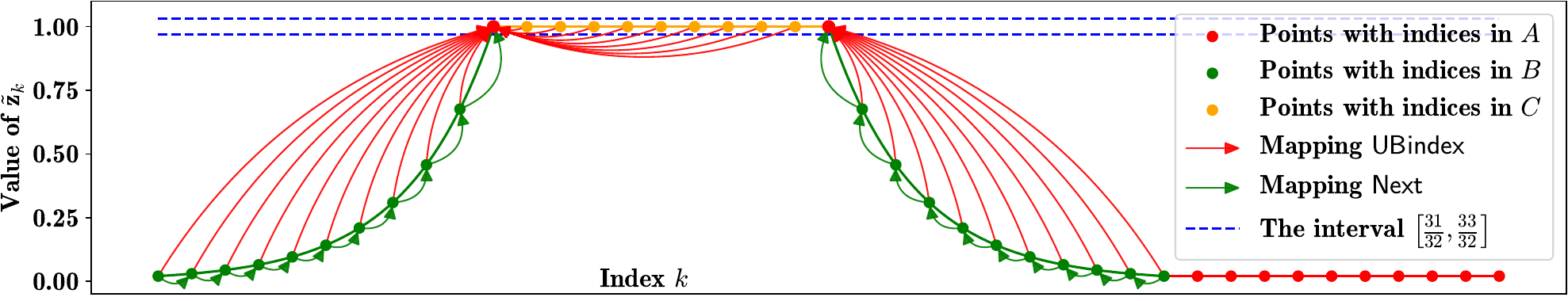}  
        \caption{An example showing the intuition of our notations and how to estimate the PL constant.}
        \label{fig:illustration}
        \end{figure}
        
        We first introduce some notations to simplify the proofs. We introduce $\tilde \x$ and $\tilde \z$ in \eqref{equ:eq37}, \eqref{equ:eq38} and \eqref{equ:eq39} to show that the ``exponential growing chains'' will terminate at the beginning and end of $\{\z_k\}_{k=1}^{Tt}$. It also simplifies the Lemmas we introduce later. For $\x\in \R^{Tt}$, let
        \begin{equation}
            \tx = \begin{bmatrix}0\\
            \x\\
            \x_{Tt}
            \end{bmatrix}\in \R^{Tt+2},
            \label{equ:eq37}
        \end{equation}
        \begin{equation}
            \tB = \begin{bmatrix}-\ee{1}&\bB+\EE{Tt,Tt}&-\ee{Tt}\end{bmatrix}\in \R^{Tt\times(Tt+2)},
            \label{equ:eq38}
        \end{equation}
        and
        \begin{equation}
            \tI = \begin{bmatrix}\mathbf 0&\bI_{Tt}&\mathbf 0\end{bmatrix}\in \R^{Tt\times(Tt+2)},
            \label{equ:eq39}
        \end{equation}
        where $\ee{i}$ is the $i$th \modify
        {column} of $\bI$\modify{, and $\EE{Tt,Tt}$ is an $Tt\times Tt$ whose $(Tt,Tt)$ entry is $1$ and other entries are $0$}. We have $\bB\x = \tB\tx$, and $\x = \tI\tx$. We use $\tx_0,\cdots,\tx_{Tt+1}$ to denote the coordinates of $\tx$, i.e. $\tx_0 = 0$, $\tx_i = \x_i$ for $i=1,\cdots, Tt$, and $\tx_{Tt+1} = \x_{Tt}$. Similarly, we define $\ty$ by $\ty_0 = 1$, $\ty_i=\y_i$ for $i=1,\cdots,Tt$, and $\ty_{Tt+1} = \y_{Tt}$. With these newly defined notations, we can check that
        \begin{equation}
            \begin{aligned}
            \left(\nabla \gtt(\x)\right)_k &= (\tB\tx)_k +\tx_k- b_{\ty_k}(\tx_k)\\
            \end{aligned}
        \end{equation}
        for \modify{$k=1,2,\cdots,Tt$}.

        We define $C_3 = \frac{21344400}{1083}$. Let $\mu$ be the PL constant of $\gtt$. 
        \begin{equation}
            \begin{aligned}
            \mu &= \inf_{\x\in \R^{Tt}} \frac{\|\nabla \gtt(\x)\|^2}{2(\gtt(x)-\gtt^*)}\\
            &= \inf_{\x\in\R^{Tt}}\frac{\|\nabla \gtt(\x)\|^2}{2\gtt(\x)}\\
            &\stackrel{(a)}{\ge} \inf_{\x\in\R^{Tt}}\frac{\|(\bB+\bI)\x-\bbT(\x)\|^2}{\x^T(\bB+\bI)\x}\\
            &\stackrel{(b)}{\ge} \inf_{\x\in\R^{Tt}}\frac{\|(\bB+\bI)\x-\bbT(\x)\|^2}{5\|\x\|^2}\\ 
            \end{aligned}
        \end{equation}
        where $(a)$ follows from property \ref{gprop1:min} of Lemma \ref{lem:gttprop}, and $(b)$ follows from $\bB+\bI \preceq 5\bI$. Define 
        \begin{equation}
            \tilde{\mu}(\x) =\frac{\|\nabla \gtt(\x)\|^2}{5\|\x\|^2}= \frac{\|(\bB+\bI)\x-\bbT(\x)\|^2}{5\|\x\|^2}.
        \end{equation}
        We only need to give a lower bound on $\tilde{\mu}$. We estimate the proportion \\$\frac{\|\nabla \gtt(\x)\|}{\|\tx_k\|} = \frac{\|(\bB+\bI)\x-\bbT(\x)\|}{\|\x\|}$ by computing each $\left|\frac{(\nabla \gtt(\x))_k}{\tx_k} \right| = \left|\frac{((\bB+\bI)\x)_k-b_{\ty_k}(\tx_k)}{\tx_k} \right|$. If $\left|\frac{(\nabla \gtt(\x))_k}{\tx_k} \right|$ is small, we will upper bound $\tx_k^2$ by one of the nearby $(\nabla\gtt(\x))_n^2$ terms.
        
        We define an operator $\dominate: \{1,\cdots, Tt\}\to \{\true, \false\}$ as follows:
        \begin{equation}
            \dominate(k)=\begin{cases}
                \true, & \mathrm{If }\left|\frac{(\nabla \gtt(\x))_k}{\tx_k} \right| > 0.19 \text{ or } \tx_k=0,\\
                \false, & \mathrm{otherwise}.
            \end{cases}
        \end{equation}
        Define $\tz_k = \frac{\tx_k}{\ty_k}$, we present four auxiliary lemmas below. We prove them in Section \ref{subsec:auxproof}.
        \newcounter{app7}
        \setcounter{app7}{\value{theorem}}
        \begin{lemma}
            For $k=1,\cdots,Tt$, if $\tz_k\notin [\frac{31}{32}, \frac{33}{32}]$ and $\neg \dominate(k)$, then $\frac{\tz_{k-1}}{\tz_k}\ge \frac{4}{3}$ or $\frac{\tz_{k+1}}{\tz_k}\ge \frac{4}{3}$.
            \label{lem:dom}
        \end{lemma}
        \newcounter{app8}
        \setcounter{app8}{\value{theorem}}
        \begin{lemma}
            If $\tz_{k-1}<\frac{5}{7}$, $\tz_{k}\in[\frac{31}{32}, \frac{33}{32}]$, $\tz_{k+1}\le \frac{33}{32}$, then $\dominate(k)=\true$.
            \label{lem:into1}
        \end{lemma}
        \newcounter{app9}
        \setcounter{app9}{\value{theorem}}
        \begin{lemma}
            If $\tz_{k-2}<\frac{31}{32}$, $\frac{5}{7}\le \tz_{k-1}<\frac{31}{32}$, $\tz_{k}\in [\frac{31}{32}, \frac{33}{32}]$, then $\dominate(k-1)=\true$.
            \label{lem:into2}
        \end{lemma}
        \newcounter{appA}
        \setcounter{appA}{\value{theorem}}
        \begin{lemma}
            If $\tz_n\in[\frac{31}{32},\frac{33}{32}]$, $\neg\dominate(n)$ and $\tz_{n-1}, \tz_{n+1}\le \frac{33}{32}$, there exist $k$ such that $k<n$, $\dominate(k)$ and $\tz_{k}>\frac{5}{7}$.
            \label{lem:dom2}
        \end{lemma}
        Note that if we alter the ``$+$'' and ``$-$'' sign in Lemma \ref{lem:into1} or \ref{lem:into2}, the conclusion still holds. 

        Now we define an operator $\Next$ on indices on which $\dominate$ operator is false. Intuitively, $\Next$ finds the direction in which $\tz$ grows exponentially.

        By Lemma \ref{lem:dom}, if $\tz_k\notin [\frac{31}{32},\frac{33}{32}]$ and $\neg\dominate(k)$, define $\Next(k)\in \{k-1,k+1\}$ to be one of the coordinate satisfying $\frac{\tz_{\Next(k)}}{\tz_k}\ge \frac{4}{3}$.  
        
        Next, we define how the $\Next$ operator acts on the index $n$ where $\tz_n\in [\frac{31}{32},\frac{33}{32}]$ and \\$\max\left\{\tz_{n-1},\tz_{n+1} \right\}> \frac{33}{32}$. Without the loss of generality (if we alter the ``$+$'' and ``$-$'', the following conclusion still holds), if $\tz_{n+1}>\frac{33}{32}$, we define $\Next(n) = n+1$. If $\Next(n-1)=n$, $\tz_n\in [\frac{31}{32}, \frac{33}{32}]$ and $\neg\dominate(n)$, then by Lemma \ref{lem:into1}, we have $\tz_{n+1}\ge \frac{33}{32}$. Therefore, if $n=\Next(m)$, then $\Next(n)$ is defined. We can apply the $\Next$ operator recursively, and will finally reach a index $n$ such that $\dominate(n) = \true$. This process will terminate because $\tz_0 = 0$ and $\tz_{Tt+1} = \tz_{Tt}$, ensuring that if the recursive $\Next$ operation reaches $0$ or $Tt$, it terminates.
        
        For other $n$ such that $\neg\dominate(n)$ and $\tz_n\in [\frac{31}{32},\frac{33}{32}]$ ($\tz_{n-1}, \tz_{n+1}\le \frac{33}{32}$), the operator $\Next$ is undefined. We will use Lemma \ref{lem:dom2} to tackle this situation.

        For $n$ such that $\neg\dominate(n)$, we define an operator $\UBindex$, and use a proportion of $\|\gtt(\x)\|_{\UBindex(n)}$ to upper bound $\x_n$. The process of finding $\UBindex$ is provided in Algorithm \ref{alg:UB}.

        \begin{algorithm}[t]
            \caption{Algorithm to find $\UBindex(n)$}\label{alg:UB}
            \begin{algorithmic}[1]
            \Require $\neg\dominate(n)$ 
            \If{$\Next(n)$ is defined}
                \State{$m\gets n$}
                \While{$\neg\dominate(m)$}
                    \State{$m\gets\Next(m)$}\Comment{This process is well-defined and will terminate.}
                \EndWhile
                \State{$\UBindex(n)\gets m$}
            \ElsIf{$\Next(n)$ is not defined}
                \State{$\UBindex(n)\gets\argmax_{k<n} \left\{\tz_k:\dominate(k) \right\}$.}
                \Comment{$\UBindex(n)$ exists and $\tz_{\UBindex(n)} > \frac{5}{7}$ (by Lemma \ref{lem:dom2}).}
            \EndIf
            \end{algorithmic}
        \end{algorithm}
        Define $\Dist(n)=|n-\UBindex(n)|$, and define $\UB(k)$ to be a proportion of $\|\gtt(\x)\|_{\UBindex(k)}^2$ as follows:
        \begin{equation}
            \UB(n) = \begin{cases}
                \frac{1}{4}|(\nabla\gtt(\x))_n|^2, \\
                \qquad\qquad\qquad\qquad\qquad\dominate(n),\\
                \frac{1}{4}\cdot\frac{13}{49}\cdot\left(\frac{36}{49}\right)^{\Dist(n)-1}|(\nabla\gtt(\x))_{\UBindex(n)}|^2, \\
                \qquad\qquad\qquad\qquad\qquad\neg\dominate(n),\Next(n) \text{ exists,}\\
                \frac{1}{4T}\cdot\frac{15}{64}\cdot\frac{\ty_n^2}{\ty_{\UBindex(n)}^2}|(\nabla\gtt(\x))_{\UBindex(n)}|^2, \\
                \qquad\qquad\qquad\qquad\qquad\neg\dominate(n),\Next(n) \text{ does not exist}.
            \end{cases}
        \end{equation}
        Let $A = \{n:\dominate(n)\}$, $B = \{n:\neg\dominate(n), \Next(n)\text{ exists}\}$, and $C = \{n:\neg\dominate(n), \\\Next(n) \text{ does not exist}\}$. Now we calculate $\sum_{i=1}^{Tt} \UB(n)$, and show that it is smaller than $\|\gtt(\x)\|^2$.
        \begin{equation}
            \begin{aligned}
                \sum_{i=1}^{Tt} \UB(n) &= \sum_{n\in A}\UB(n) + \sum_{n\in B}\UB(n)+\sum_{n\in C}\UB(n)\\
                &= \frac{1}{4}\sum_{n\in A} |(\nabla\gtt(\x))_n|^2 \\
                &\quad +\sum_{n\in B} \frac{13}{196}\left(\frac{36}{49}\right)^{\Dist(n)-1}|(\nabla\gtt(\x))_{\UBindex}|^2\\
                &\quad +\sum_{n\in C} \frac{15}{256T}\left(\frac{\ty_n}{\ty_{\UBindex(n)}}\right)^{2}|(\nabla\gtt(\x))_{\UBindex}|^2.
            \end{aligned}
            \label{equ:division}
        \end{equation}
        By changing the order of summation, we have
        \begin{equation}
            \begin{aligned}
                &\quad \sum_{n\in B} \frac{13}{196}\left(\frac{36}{49}\right)^{\Dist(n)-1}|(\nabla\gtt(\x))_{\UBindex}|^2 \\
                &\le \frac{13}{196} \sum_{n\in A} 2\cdot\sum_{i=1}^{\infty} \left(\frac{36}{49}\right)^{i-1}|(\nabla\gtt(\x))_n|^2\\
                &= \frac{1}{2}\sum_{n\in A}|(\nabla \gtt(\x))_n|^2,
            \end{aligned}
            \label{equ:Bsum}
        \end{equation}
        and 
        \begin{equation}
            \begin{aligned}
                &\quad \sum_{n\in C} \frac{15}{256T}\left(\frac{\ty_n}{\ty_{\UBindex(n)}}\right)^{2}|(\nabla\gtt(\x))_{\UBindex}|^2\\
                &\le \frac{15}{256T}\sum_{n\in A}\sum_{m\ge n}\left(\frac{\ty_m}{\ty_n}\right)^2 |(\nabla \gtt(\x))_n|^2\\
                &= \frac{15}{256T}\sum_{n\in A}T\sum_{i=0}^\infty\left(\frac{7}{8}\right)^{2i} |(\nabla \gtt(\x))_n|^2\\
                &= \frac{1}{4}\sum_{n\in A}|(\nabla g(\x))_n|^2.
            \end{aligned}
            \label{equ:Csum}
        \end{equation}
        Summing up \eqref{equ:division}, \eqref{equ:Bsum} and \eqref{equ:Csum}, we have
        \begin{equation}
            \begin{aligned}
                \sum_{i=1}^{Tt}\UB(n) &= \frac{1}{4} \sum_{n\in A}|(\nabla \gtt(\x))_n|^2 + \frac{1}{2} \sum_{n\in A}|(\nabla \gtt(\x))_n|^2 + \frac{1}{4} \sum_{n\in A}|(\nabla \gtt(\x))_n|^2 \\
                &= \sum_{n\in A}|(\nabla \gtt(\x))_n|^2 \\
                &\le \|\nabla \gtt(\x)\|^2.
            \end{aligned}
        \end{equation}
        Finally, we calculate $\frac{\UB(n)}{\x_n^2}$ to give an universal lower bound of the PL constant. For $n\in A$, we have 
        \begin{equation}
            \begin{aligned}
                \frac{\UB(n)}{\x_n^2}&= \frac14\cdot\frac{|(\nabla \gtt(\x))_n|^2}{\x_n^2}\\
                &\stackrel{(a)}{\ge}\frac{361}{40000},\\
            \end{aligned}
        \end{equation}
        where $(a)$ is due to $\dominate(n)=\true$. 
        
        For $n\in B$, by Lemma \ref{lem:dom} we have $\left|\frac{\tz_{\Next(m)}}{\tz_{m}} \right|\ge\frac{4}{3}$ for $m=n,\Next(n),\Next(\Next(n))\cdots$, with only one possible exception when $\tz_m\in [\frac{31}{32},\frac{33}{32}]$, in which case there is $|\tz_{\Next(m)}|\ge|\tz_{m}|$. Therefore, $|\tz_{\UBindex(n)}|\ge\left(\frac{4}{3}\right)^{\Dist(n)-1}|\tz_n|$, so we have
        \begin{equation}
            \begin{aligned}
                \frac{\UB(n)}{\x_n^2}&= \frac{1}{4}\cdot\frac{13}{49}\cdot\left(\frac{36}{49}\right)^{\Dist(n)-1}|(\nabla\gtt(\x))_{\UBindex(n)}|^2\\
                &\stackrel{(a)}{\ge} \frac{13}{196}\cdot \left(\frac{36}{49}\right)^{\Dist(n)-1} \frac{361}{10000}\frac{\x_{\UBindex(n)}^2}{\x_n^2}\\
                &\stackrel{(b)}{=} \frac{4693}{1960000}\left(\frac{36}{49}\right)^{\Dist(n)-1}\cdot\left(\frac{\tz_{\UBindex(n)}}{\tz_n}\right)^2\cdot\left(\frac{\ty_{\UBindex(n)}}{\ty_n} \right)^2\\
                &\stackrel{(c)}{\ge}\frac{4693}{1960000}\left(\frac{36}{49}\right)^{\Dist(n)-1}\cdot \left(\frac{16}{9}\right)^{\Dist(n)-1}\cdot\left(\frac{49}{64}\right)^{\Dist(n)}\\
                &= \frac{4693}{2560000},
            \end{aligned}
        \end{equation}
        where $(a)$ is due to $\dominate(\UBindex(n))=\true$, $(b)$ is due to $\tz_m = \frac{\tx_m}{\ty_m}$, $(c)$ is due to $|\tz_{\UBindex(n)}|\ge\left(\frac{4}{3}\right)^{\Dist(n)-1}|\tz_n|$ and $\frac{|\ty_{m+1}|}{|\ty_m|}\ge \frac{7}{8}$.

        For $n\in C$, we have
        \begin{equation}
            \begin{aligned}
                \frac{\UB(n)}{\x_n^2} &= \frac{1}{4T}\cdot\frac{15}{64}\cdot\frac{\ty_n^2}{\ty_{\UBindex(n)}^2}\cdot\frac{|(\nabla\gtt(\x))_{\UBindex(n)}|^2}{\x_n^2}\\
                &\stackrel{(a)}{\ge}\frac{15}{256T}\cdot\frac{\ty_n^2}{\ty_{\UBindex(n)}^2}\cdot\frac{361\tx_{\UBindex(n)}^2}{40000\tx_n^2}\\
                &\stackrel{(b)}{\ge} \frac{1083}{2048000T}\frac{\tz_{\UBindex(n)}^2}{\tz_n^2}\\
                &\stackrel{(c)}{\ge}\frac{1083}{2048000T}\left(\frac{\sfrac{5}{7}}{\sfrac{33}{32}}\right)^2\\
                &=\frac{1083}{4268880T},
            \end{aligned}
        \end{equation}
    \end{enumerate}
    where $(a)$ is due to $\dominate(\UBindex(n))=\true$, $(b)$ is due to $\tz_m = \frac{\tx_m}{\ty_m}$, and $(c)$ is due to Lemma \ref{lem:dom2} and $\tz_n\in[\frac{31}{32},\frac{33}{32}]$.
    Therefore, 
    \begin{equation}
        \begin{aligned}
            \mu&\ge \inf_{\x}\frac{\|\nabla\gtt(\x)\|^2}{5\x^2}\\
            &\ge \inf_{\x}\frac{\sum_{n=1}^{Tt}\UB(n)}{5\sum_{n=1}^{Tt}\x_n^2}\\
            &\ge \inf_{\x}\min_{n} \frac{\UB(n)}{5\x_n^2}\\
            &\ge \inf_{\x} \frac{1}{5}\cdot\min\left\{\frac{361}{40000}, \frac{4693}{2560000}, \frac{1083}{4268880T}\right\}\\
            &= \frac{1083}{21344400T}.
        \end{aligned}
    \end{equation}
\end{proof}

\subsection{Proof of Lemma \ref{lem:dom}, \ref{lem:into1}, \ref{lem:into2} and \ref{lem:dom2}}
\label{subsec:auxproof}

\newcounter{temp7}
\setcounter{temp7}{\value{theorem}}
\setcounter{theorem}{\value{app7}}
\begin{lemma}
    For $k=1.\cdots.Tt$, if $\tz_k\notin [\frac{31}{32}, \frac{33}{32}]$ and $\neg \dominate(k)$, then $\frac{\tz_{k-1}}{\tz_k}\ge \frac{4}{3}$ or $\frac{\tz_{k+1}}{\tz_k}\ge \frac{4}{3}$.
\end{lemma}
\begin{proof}
    If $\tz_k\notin [\frac{31}{32},\frac{33}{32}]$, then $b_{\ty_k}(\tx_k) =  0$, and $(\nabla \gtt(\x))_k = (\tB\tx)_k +\tx_k$. By $\neg\dominate(k)$, we have $\left|{(\tB\tx)_k +\tx_k}\right|<0.19\left|\tx_k\right|$. We consider three cases:
    \begin{enumerate}
        \item If $k = T+1,2T+1,\cdots, T(t-1)+1$, $(\tB\tx)_k = -\frac{7}{8} \tx_{k-1}+2\tx_k-\tx_{k+1}$, $(\tB\tx)_k + \tx_k = \frac{7}{8} \tx_{k-1}+3\tx_k-\tx_{k+1}$, and $\frac{7}{8}\ty_{k-1} =\ty_k=\ty_{k+1} $. We have
        \begin{equation}
            \begin{aligned}
                3|\tx_k|-\frac{7}{8}|\tx_{k-1}|-|\tx_{k+1}| &\le |(\tB\tx)_k + \tx_k | \\
                &< 0.19|\tx_k|.
            \end{aligned}
        \end{equation}
        Dividing both sides by $\ty_k$, we have
        \begin{equation}
            3|\tz_k|-|\tz_{k-1}|-|\tz_{k+1}|<0.19|\tz_k|.
            \label{equ:notdom}
        \end{equation}
        Thus, we have $|\tz_{k-1}|+\tz_{k+1}>\frac{8}{3}|\tz_k|$, which indicates that $\frac{\tz_{k-1}}{\tz_k}\ge \frac{4}{3}$ or $\frac{\tz_{k+1}}{\tz_k}\ge \frac{4}{3}$.
        \item If $k= T, 2T,\cdots, T(t-1)$, $(\tB\tx)_k+\tx_k = -\tx_{k-1}+\frac{177}{64}\tx_k-\frac{7}{8}\tx_{k+1}$, and $\ty_{k-1}=\ty_k=\frac{8}{7}\ty_{k+1}$. Thus, we have
        \begin{equation}
            \frac{177}{64}|\tx_{k}|-|\tx_{k-1}|-\frac{7}{8}|\tx_{k+1}|<0.19|\tx_{k}|.    
        \end{equation}
        Dividing both sides by $\ty_k$, we have
        \begin{equation}
            \frac{177}{64}|\tz_k|-|\tz_{k-1}|-\frac{49}{64}|\tz_{k+1}|<0.19|\tz_{k}|.
        \end{equation}
        Thus, we have
        \begin{equation}
            |\tz_{k-1}|+\frac{49}{64}|\tz_{k+1}| > \left(\frac{177}{64}-0.19\right)|\tz_{k}|>\frac{4}{3}\cdot \left(1+\frac{49}{64}\right)|\tz_k|.
        \end{equation}
        \item For other $k$, $(\tB\tx)_k+\tx_k = -\tx_{k-1}+3\tx_{k}-\tx_{k+1}$, and $\ty_{k-1}=\ty_k=\ty_{k+1}$. Therefore, we have (by dividing $\ty_k$ to both sides of \eqref{equ:notdom})
        \begin{equation}
            |\tz_{k-1}|+|\tz_{k+1}|>2.81|\tz_{k}|>\frac{8}{3}|\tz_k|.    
        \end{equation}
    \end{enumerate}
\end{proof}

\setcounter{theorem}{\value{app8}}
\begin{lemma}
    If $\tz_{k-1}<\frac{5}{7}$, $\tz_{k}\in[\frac{31}{32}, \frac{33}{32}]$, $\tz_{k+1}\le \frac{33}{32}$, then $\dominate(k)=\true$.

\end{lemma}
\begin{proof}
    For any $y>0$, we have $0\le \by(x)\le y$. Therefore, $(\nabla\gtt(\x))_k\ge(\tB\tx)_k$.
    \begin{enumerate}
        \item If $k = T+1,2T+1,\cdots, T(t-1)+1$, $(\tB\tx)_k = -\frac{7}{8} \tx_{k-1}+2\tx_k-\tx_{k+1}$ and $\frac{7}{8}\ty_{k-1} =\ty_k=\ty_{k+1} $.
        \begin{equation}
            \begin{aligned}
                \frac{(\nabla\gtt(\x))_k}{\tx_k}&\ge \frac{(\tB\tx)_k}{\tx_k}\\
                &= \frac{-\frac{7}{8}\tx_{k-1}+2\tx_k=\tx_{k+1}}{\tx_k}\\
                &\stackrel{(a)}{=}\frac{-\tz_{k-1}+\tz_k-\tz_{k+1}}{\tz_k}\\
                &\stackrel{(b)}{>}-\frac{\sfrac{5}{7}}{\sfrac{31}{32}}+2-\frac{\sfrac{33}{32}}{\sfrac{31}{32}}\\
                &>0.19,
            \end{aligned}
        \end{equation}
        where $(a)$ holds by dividing $\ty_k$ to the numerator and denominator, $(b)$ holds by the assumptions on $\tz_{k-1},\tz_k,\tz_{k+1}$.
        \item If $k= T, 2T,\cdots, T(t-1)$, $(\tB\tx)_k = -\tx_{k-1}+\frac{113}{64}\tx_k-\frac{7}{8}\tx_{k+1}$, and $\ty_{k-1}=\ty_k=\frac{8}{7}\ty_{k+1}$.
        \begin{equation}
            \begin{aligned}
                \frac{(\nabla\gtt(\x))_k}{\tx_k}&\ge \frac{(\tB\tx)_k}{\tx_k}\\
                &= \frac{-\tx_{k-1}+\frac{113}{64}\tx_k-\frac{7}{8}\tx_{k+1}}{\tx_k}\\
                &\stackrel{(a)}{=}\frac{-\tz_{k-1}+\frac{113}{64}\tz_k-\frac{49}{64}\tz_{k+1}}{\tz_k}\\
                &\stackrel{(b)}{>}-\frac{\sfrac{5}{7}}{\sfrac{31}{32}}+\frac{113}{64}-\frac{49}{64}\cdot\frac{\sfrac{33}{32}}{\sfrac{31}{32}}\\
                &>0.19,
            \end{aligned}
        \end{equation}
        where $(a)$ holds by dividing $\ty_k$ to the numerator and denominator, and $(b)$ holds by the assumptions on $\tz_{k-1},\tz_k$, and $\tz_{k+1}$.
        \item For other $k$, $(\tB\tx)_k = -\tx_{k-1}+2\tx_k-\tx_{k+1}$, and $\ty_{k-1}=\ty_k=\ty_{k+1}$. Therefore,
        \begin{equation}
            \begin{aligned}
                \frac{(\nabla\gtt(\x))_k}{\tx_k}&\ge \frac{(\tB\tx)_k}{\tx_k}\\
                &= \frac{-\tx_{k-1}+2\tx_k=\tx_{k+1}}{\tx_k}\\
                &\stackrel{(a)}{=}\frac{-\tz_{k-1}+\tz_k-\tz_{k+1}}{\tz_k}\\
                &\stackrel{(b)}{>}-\frac{\sfrac{5}{7}}{\sfrac{31}{32}}+2-\frac{\sfrac{33}{32}}{\sfrac{31}{32}}\\
                &>0.19,
            \end{aligned}
        \end{equation}
        where $(a)$ holds by dividing $\ty_k$ to the numerator and denominator, and $(b)$ holds by the assumptions on $\tz_{k-1},\tz_k$, and $\tz_{k+1}$.
    \end{enumerate}
\end{proof}

\setcounter{theorem}{\value{app9}}
\begin{lemma}
    If $\tz_{k-2}<\frac{31}{32}$, $\frac{5}{7}\le \tz_{k-1}<\frac{31}{32}$, $\tz_{k}\in [\frac{31}{32}, \frac{33}{32}]$, then $\dominate(k-1)=\true$.
\end{lemma}
\begin{proof}
    Like the proof of Lemma \ref{lem:into1}, we directly compute $\frac{(\nabla\gtt(\x))_{k-1}}{\tx_{k-1}}$. Because $\tz_{k-1}<\frac{31}{32}$, $b_{\ty_{k-1}}(\tx_{k-1})=0$.
    \begin{equation}
        \begin{aligned}
            \frac{(\nabla\gtt(\tx))_{k-1}}{\tx_{k-1}}&= \frac{(\tB\tx)_{k-1}+\tx_{k-1}}{\tx_{k-1}}\\
            &\stackrel{(a)}{=} -a\frac{\tz_{k-2}}{\tz_{k-1}}+b-c\frac{\tz_{k}}{\tz_{k-1}},
            \label{equ:auxinto1}
        \end{aligned}
    \end{equation}
    where $(a)$ holds by dividing $\ty_k$ to the numerator and denominator. For $k-1=T,2T,\cdots,T(t-1)$, $a=-1$, $b-\frac{177}{64}$ and $c=\frac{49}{64}$. For other $k$, $a=c=1$ and $b=3$. By the assumptions on $\tz_{k-2},\tz_{k-1}$ and $\tz_k$, we have $\frac{\tz_{k-2}}{\tz_{k-1}}<\frac{\sfrac{31}{32}}{\sfrac{5}{7}}$, and $\frac{\tz_k}{\tz_{k-1}}<\frac{\sfrac{33}{32}}{\sfrac{5}{7}}$. Plugging everything into \eqref{equ:auxinto1}, we have $\frac{(\nabla\gtt(\tx))_{k-1}}{\tx_{k-1}} > 0.2$.
\end{proof}

\setcounter{theorem}{\value{appA}}
\begin{lemma}
    If $\tz_n\in[\frac{31}{32},\frac{33}{32}]$, $\neg\dominate(n)$ and $\tz_{n-1}, \tz_{n+1}\le \frac{33}{32}$, there exist $k$ such that $k<n$, $\dominate(k)$ and $\tz_{k}>\frac{5}{7}$.

\end{lemma}
\setcounter{theorem}{\value{temp7}}

\begin{proof}
    Define $m$ = $\argmax_{m'}\left\{\tz_{m'}\notin [\frac{31}{32},\frac{33}{32}]:m'<n \right\}$. If $\tz_{m}>\frac{33}{32}$, let $k$ be $m$ if $\dominate(m)$ and $\UBindex(m)$ if $\neg\dominate(m)$. If $\tz_m\in [\frac{5}{7},\frac{33}{32}]$, let $k$ be $m$ if $\dominate(m)$ and $\UBindex(m)$ if $\neg\dominate(m)$. Finally, if $\tz_m<\frac{5}{7}$, let $k=m+1$.
\end{proof}

\fi



\end{document}